\documentclass{amsart}
\usepackage[utf8]{inputenc}

\title{$\Sigma^1_3$ sets in the Sacks model}
\author{Jonathan Schilhan}
\email{jonathan.schilhan@univie.ac.at}
\urladdr{http://www.logic.univie.ac.at/~jschilhan/}
\address{University of Vienna,
Institute of Mathematics,
Kurt Gödel Research Center,
Kolingasse 14-16,
1090 Vienna,
Austria}
\date{\today}

\usepackage{amsmath}
\usepackage{amssymb}
\usepackage{mathrsfs}
\usepackage{tikz-cd}
\usepackage{enumitem}
\usepackage{url}
\usepackage{hyperref}
\usepackage{bbm}
\usepackage{ stmaryrd }
\usepackage{todonotes}

\usepackage{accents}

    \DeclareMathOperator{\Ss}{\mathbb{S}}
    
    \DeclareMathOperator{\OD}{OD}

    \DeclareMathOperator{\cf}{cf}

\newcommand{\axiom}[1]{\mathsf{#1}}

\newcommand{\CH}{\axiom{CH}}
\newcommand{\MS}{\axiom{MS}}

\newcommand{\ZF}{\axiom{ZF}}

\newcommand{\WO}{\axiom{WO}}

\theoremstyle{plain}% default
\newtheorem{thm}{Theorem}[section]
\newtheorem{lemma}[thm]{Lemma}
\newtheorem{prop}[thm]{Proposition}
\newtheorem{cor}[thm]{Corollary}
\newtheorem{claim}[thm]{Claim}
\newtheorem{fact}[thm]{Fact}
\theoremstyle{definition}
\newtheorem{definition}[thm]{Definition}

\newtheorem{remark}[thm]{Remark}
\newtheorem{quest}[thm]{Question}

\begin{document}

\begin{abstract}
    We show that in the iterated Sacks model over the constructible universe the Mansfield-Solovay Theorem holds for $\Sigma^1_3$ sets. In particular, every $\mathbf{\Sigma}^1_3$ set is Marczewski measurable and the optimal complexity for a Bernstein set is $\Delta^1_4$. Based on a result by Kanovei, we also briefly show how to separate the Mansfield-Solovay Theorem at non-trivial levels of the projective hierarchy.
\end{abstract}

\maketitle

\section{Introduction}

For the purposes of this article a \emph{pointclass} will be a collection $\Gamma$ of subsets of the Polish spaces $(\omega^\omega)^n$ for $n \geq 1$. Whenever $B \subseteq (\omega^\omega)^{n+1}$ for some $n \geq 1$ and $x \in \omega^\omega$, the $x$'th section of $B$ is the set $$B_x = \{ (y_0, \dots, y_{n-1}) \in (\omega^\omega)^n : (x,y_0, \dots, y_{n-1}) \in B \}.$$ To any pointclass $\Gamma$ and $x \in \omega^\omega$ we associate the \emph{parameterized pointclass} $\Gamma(x) = \{ B_x  : B \in \Gamma\}$ and the \emph{boldface pointclass} $\mathbf{\Gamma} = \bigcup_{x \in \omega^\omega} \Gamma(x)$. We will be mainly interested in the projective classes  $\Sigma^1_n$, $\Pi^1_n$ and $\Delta^1_n$ for $n = 1, 2, 3, 4$. For a more general treatment and more information about pointclasses we refer to \cite{Moschovakis2009}.

\begin{definition}
    Let $\Gamma$ be a pointclass. Then the Mansfield-Solovay Theorem (MST) for $\Gamma$ is denoted $\MS(\Gamma)$ and says that for every $r \in \omega^\omega$ and $A \in \Gamma(r)$, at least one of the following is satisfied: 

    \begin{enumerate}
        \item $A \subseteq L[r]$.
        \item $A$ contains a (non-empty) perfect subset. 
    \end{enumerate}
\end{definition}

Recall that a perfect subset of a topological space is a closed set without isolated points. As usual in the literature, we will always consider perfect sets to additionally be non-empty, whence the brackets in (2) above.

The classical Mansfield-Solovay Theorem (see e.g. \cite[Theorem 25.23]{Jech}), a major cornerstone of descriptive inner-model theory, concerns the pointclass $\Sigma^1_2$ and can be expressed as follows.

\begin{thm}[Mansfield-Solovay]
$\MS(\Sigma^1_2)$ holds in $\ZF$.
\end{thm}

Of course, in the constructible universe itself the MST is trivial and holds for the pointclass of all sets of reals. On the other hand, whenever $r \notin L$, $\{r\}$ neither has a perfect subset, nor is a subset of $L$ and this extreme end of the MST must fail. This is not exactly fair, as in choosing the pointclass of all sets, we have simply disregarded $r$ as a parameter in the definition of its singleton. 

Rather, we should consider some notion of a \emph{lightface pointclass} for the theorem to be useful. Still, it is in general not hard to obtain strict forcing extensions of $L$ in which $\MS(\OD)$ holds, where $\OD$ is the pointclass of all ordinal definable sets of reals. For instance, this is the case after adding uncountably many Cohen reals.\footnote{Use the homogeneity of Cohen forcing and the fact that it adds a perfect set of Cohen reals.} On the other hand, we do not know what happens after adding a single Cohen real (see Question~\ref{q:Cohen}). $\OD$ is in a sense the largest sensible lightface pointclass.

In the other direction, note that when $0^\sharp$ exists, necessarily $\MS(\Pi^1_2)$ must fail, as $\{0^\sharp\}$ is $\Pi^1_2$ but not contained in $L$. Moreover, this persists through any larger model (with the same ordinals). For this reason, it makes sense in general to consider a refinement of the MST relative to any ``ground model", not necessarily equal to $L$. This will be useful to state our results in a slightly more general form.

\begin{definition}
    Let $\Gamma$ be a pointclass, $M$ an inner model. We write $\MS(\Gamma, M)$ to say that for any $r \in \omega^\omega$ and $A \in \Gamma(r)$, one of the following is satisfied: 

    \begin{enumerate}
        \item $A \subseteq M[r]$.
        \item $A$ contains a perfect set.
    \end{enumerate}
\end{definition}

Even without the use of large cardinals, it is possible to have non-constructible $\Pi^1_2$ singletons and thus obtain a failure of $\MS(\Pi^1_2)$, as shown by Jensen (see \cite{Jensen1970}). Yet, there is an even simpler way towards the same result: add any number of random reals over $L$. Being a random real over $L$ is a $\Pi^1_2$ property. On the other hand, random forcing does not add a perfect set of random reals (see \cite{Bartoszynski1990}).

The strong discrepancy between the Cohen and the random models begs the question of what interesting intermediate results we can obtain in other well-known forcing extensions. 

The focus of our paper is therefore on the \emph{Sacks model}, or rather the several \emph{iterated Sacks models}, which are obtained using countable support iterations of \emph{Sacks forcing} (see Definition~\ref{def:Sacksforcing}). These models are interesting for several reasons. On the one hand, the iteration of length $\omega_2$ (over a model of $\CH$) provides a more or less canonical model for $\neg\CH$ in which many of the cardinal invariants of the continuum simultaneously take value $\aleph_1$. This is somewhat confirmed by Zapletal's theory of \emph{tame invariants}, see \cite{Zapletal2008}. On the other hand, they have been intensely studied for their interesting real degree structure, foremost by Groszek \cite{Groszek1988}, also Kanovei \cite{Kanovei1999} (more on that in Section~\ref{sec:Sacks}). Budinas \cite{Budinas1983} has shown that in the Sacks model over $L$ every analytic equivalence relation has a $\mathbf{\Delta}^1_2$ transversal, that is, a set with exactly one element from each equivalence class. This indicates how the Sacks model is also worth studying for its descriptive properties. In \cite{Schilhan2022a}, we succeeded to extend Budinas' result to $\mathbf{\Delta}^1_2$ maximal independent sets in analytic hypergraphs. 

The following is our main result.

\begin{thm}\label{thm:main}
    Let $\mathbb{P}$ be a countable support iteration of Sacks forcing of length $\lambda$, where $\lambda = 1$ or $\lambda$ is a limit of uncountable cofinality. Then $V^\mathbb{P} \models \MS(\Sigma^1_3, V)$.
\end{thm}

Particularly, when forcing over the constructible universe, we obtain:

\begin{thm}\label{thm:mainL}
(V=L) Let $\mathbb{P}$ be a countable support iteration of Sacks forcing of length $\lambda$. Then $V^\mathbb{P} \models \MS(\Sigma^1_3)$ exactly when $\cf(\lambda) > \omega$ or $\lambda = 1$. If $\lambda > 1$, then $V^\mathbb{P} \models\neg \MS(\Pi^1_3)$.
\end{thm}

Let us mention that the original motivation for this result, comes from a slightly different angle than the Mansfield-Solovay Theorem directly. In \cite{Schilhan2022}, we have observed that, after iterating Sacks forcing over $L$ for $\lambda$-many steps, where $\cf(\lambda) > \omega$, there is a $\Delta^1_4$-definable \emph{Bernstein set}. Theorem~\ref{thm:mainL} confirms that this is the optimal complexity for such a set in that model. Here, recall that $A \subseteq \omega^\omega$ is a Bernstein set if neither $A$ nor $\omega^\omega \setminus A$ contains a perfect set.

\begin{lemma}
    Assume that $\MS(\Gamma)$ holds and $\mathbb{R} \not\subseteq L$. Then there is no Bernstein set in $\Gamma$. If $\mathbb{R} \not\subseteq L[r]$, for every $r \in \omega^\omega$, then there is no Bernstein set in $\mathbf\Gamma$.
\end{lemma}

\begin{proof}
As $\mathbb{R} \not\subseteq L$, there is a perfect set of non-constructible reals. Suppose $A \in \Gamma$. If $A \subseteq L$, then there is a perfect subset of $\omega^\omega \setminus A$. The other option is that $A$ contains a perfect subset. In either case, $A$ is not a Bernstein set. Similarly for $\mathbf{\Gamma}$.
\end{proof}

The Mansfield-Solovay theorem is also related to a regularity property known as \emph{Marczewski measurabilty}. This is due to Brendle and Löwe \cite{BrendleLoewe} who used the MST to show that a $\Sigma^1_2(r)$ set is Marczewski measurable exactly when $\omega^\omega \setminus L[r]$ is non-empty.  We recall that a set $A \subseteq \omega^\omega$ is Marczewski measurable if for every perfect set $P \subseteq \omega^\omega$, there is a perfect subset $P' \subseteq P$ such that either $P' \subseteq A$, or $P' \cap A = \emptyset$. With exactly the same argument as in \cite{BrendleLoewe}, one concludes: 

\begin{thm}\label{thm:marcz}
Let $\mathbb{P}$ be a countable support iteration of Sacks forcing of length a limit of uncountable cofinality. Then in $V^\mathbb{P}$, every $\mathbf{\Sigma}^1_3$ set is Marczewski measurable.
\end{thm}

Part of the argument for Theorem~\ref{thm:main} includes a reflection theorem for the Sacks iteration that might be of independent interest. Roughly speaking, it corresponds to the absoluteness of $\Sigma^1_3$ formulas between the full iteration and the first $\omega_1$-many steps. Moreover, when the length of the iteration has uncountable cofinality, this can be extended up to $\Sigma^1_4$. See Theorem~\ref{thm:ref13} and ~\ref{thm:ref14} and the subsequent corollaries.

While our main focus is on a particular forcing construction, it is still interesting to consider what can be said more generally about the consistency of $\MS(\Sigma^1_n) \wedge \neg\MS(\Pi^1_n)$, for other $n$. 

In \cite{Kanovei1978}, Kanovei has shown how to extend Jensen's result to obtain, for any $n \geq 2$, a forcing extension of $L$ with a $\Pi^1_n$ singleton, where every $\Sigma^1_n$ real (that is, corresponding to a $\Delta^1_n$ singleton) is constructible. Interestingly, the properties of the forcing notions defined in \cite{Kanovei1978} can directly be used to provide an answer to our question.

\begin{thm}\label{thm:kanovei}
    For every $n \geq 2$, there is a forcing extension of $L$ satisfying $\MS(\Sigma^1_n)$, but not $\MS(\Pi^1_n)$.
\end{thm}

The rest of the article is organised as follows. We will start by reviewing Sacks forcing and its iterations, and introduce the vitally important notation around conditions that will dominate our arguments. Next, in Section~\ref{sec:reflection}, we provide the reflection arguments mentioned above. Section~\ref{sec:coding} introduces a simple coding mechanism that will be used at one point in the proof of the main result. Finally, Section~\ref{sec:main} is concerned with proving the main result. Theorem~\ref{thm:kanovei} above is given in Section~\ref{sec:Kanovei}. We close with some open questions.

\section{The Sacks model}\label{sec:Sacks}

Much of what we say in this section is folklore and can be found in the many references for Sacks forcing in the literature (e.g. \cite{BartoszynskiJudah1995, BaumgartnerLaver, Geschke, Jech, Judah1992}). Recall that a subset $T \subseteq 2^{<\omega}$ is a \emph{tree} if it is non-empty and for any $t \in T$ and $s \subseteq t$, also $s \in T$. It is called a \emph{perfect tree} if for any $t \in T$ there are \emph{incompatible} $t_0, t_1 \in T$ with $t \subseteq t_0, t_1$, where $t_0$ and $t_1$ are incompatible when neither $t_0 \subseteq t_1$, nor $t_1 \subseteq t_0$.

\begin{definition}\label{def:Sacksforcing}
    Sacks forcing $\mathbb{S}$ is the forcing notion consisting of perfect subtrees $T \subseteq 2^{<\omega}$, ordered by inclusion. We let $\dot s$ be a name for the generic Sacks real added by $\mathbb{S}$, namely $\dot s^G = \bigcup \bigcap G \in 2^\omega$, for a generic filter $G$.

    The countable support iteration of Sacks forcing of length $\lambda$ will be denoted $\mathbb{S}^{*\lambda}$. For each $\alpha < \lambda$, we write $\dot s_\alpha$ for a name for the Sacks real added in step $\alpha$ and let $\dot{\bar s}^\lambda$ be a name for the sequence of generic Sacks reals.
\end{definition}

As usual, we can identify the Sacks real, or in the case of the iteration, the sequence of Sacks reals,  with the generic filter. Thus we shall generally rather speak of the real $s$ or the sequence $\bar s = \langle s_\alpha : \alpha < \lambda \rangle$ being $\mathbb{S}$ or $\mathbb{S}^{*\lambda}$-generic. Whenever $s$ is a Sacks real over $V$, the generic filter is given precisely by the set of perfect trees $T \in V$ so that $s \in [T]$. Something similar can be said about the iteration, see Fact~\ref{fact:basic} below.

The first use of iterated Sacks forcing appears in \cite{BaumgartnerLaver} in which it is shown that $\omega_1$ is not collapsed. Today, this is usually phrased in terms of proper forcing. 

\begin{fact}
    Sacks forcing and its countable support iterations are proper and thus do not collapse $\omega_1$.
\end{fact}

As mentioned in the introduction, the models obtained by iterating Sacks forcing are interesting for their degree structure, which we will outline in the following. 

\begin{definition}
    Let $V$ be an inner model and $x, y$ be reals. Then we write $x \leq_V y$ if $x \in V(y)$ and $x \equiv_V y$ if $x \leq_V y$ and $y \leq_V x$. We write $[x]_V$ for the set $\{ y \in \omega^\omega : y \equiv_V x \}$.
\end{definition}

Clearly $\leq_V$ is a partial order and $\equiv_V$ is an equivalence relation. We call the sets $[x]_V$, the \emph{$V$-degrees} of reals. When $V = L$, we typically speak of the \emph{degrees of constructibility}.  

\begin{fact}[Degree structure, folklore]\label{fact:degrees}
    Let $\lambda$ be an ordinal and $\bar s$ be $\mathbb{S}^{*\lambda}$-generic over $V$. Let $A$ be the set of $\alpha \leq \lambda$, that are not a limit of uncountable cofinality. For each $\alpha \in A$, let $C_\alpha \in V$ be a countable cofinal subset of $\alpha$, if $\alpha$ is limit, or $C_\alpha = \{ \alpha \}$ otherwise. 
    
    Then the $V$-degrees of reals in $V[\bar s]$ are exactly $[\bar s \restriction C_\alpha]_V$, for $\alpha \in A$, and $\bar s \restriction C_\alpha \leq_V \bar s \restriction C_\beta$, whenever $\beta \leq \alpha$. In particular, they form a well-order of type $\lambda$, if $\lambda$ is a limit of uncountable cofinality, or $\lambda +1$ otherwise. 
\end{fact}

Formulated in a slightly different way, in $V[\bar s]$, there is a distinct degree for each individual Sacks real $s_\alpha$ and one for each limit of countable cofinality, which is the supremum of previous degrees. Specifically, the choice of the cofinal set $C_\alpha \subseteq \alpha$ in the fact above does not matter, as long as it is made in $V$.

As far as we can tell, the first more or less explicit mention of this in the literature is \cite[Corollary 14]{Groszek1988}. Groszek remarks that the result follows from work of A. Miller \cite{Miller1983}. In \cite{Kanovei1999}, Kanovei vastly generalizes the result to more general iterations of Sacks forcing along arbitrary partial orders (in contrast to an ordinal $\lambda$), which may also be ill-founded. Another reformulation of Fact~\ref{fact:degrees} can be found directly after Theorem 1 in \cite{Kanovei1999}.

An extremely helpful tool in the study of the countable support iteration of Sacks forcing is to consider conditions topologically. This is well-known in the context of \emph{idealized forcing} (see \cite{Zapletal2008}) and let's us consider properties such as the \emph{continuous reading of names}. There are a few approaches to this in the literature. The one we use here comes closest to Budinas' notion of normal sets from \cite{Budinas1983}, and, in our opionion, gives the most streamlined arguments due to its simplicity of notation and intuition. It can also be found in our previous work \cite{Schilhan2022, Schilhan2022a} and in some form in \cite{Kanovei1999} (see the definition of $\mathtt{Perf}_\zeta$). Also, it falls in line with Zapletal's presentation of the iteration as an idealized forcing notion (see more in \cite[Section 5.1.1]{Zapletal2008}).

\begin{definition}
    The partial order $\tilde{\mathbb{S}}^{*\lambda}$ consists of pairs $(X,C)$, where $C$ is a countable subset of $\lambda$ and $X$ is a non-empty closed subset of $(2^{\omega})^C$ such that for each $\alpha \in C$ and $\bar x \in X \restriction \alpha = \{ \bar z \restriction \alpha : \bar z \in X \}$, the set $$T_{X, \bar x} = \{ s \in 2^{<\omega} : \exists y, \bar z ( s \subseteq y \wedge \bar x^\frown y^\frown \bar z \in X)   \}$$ is a perfect tree. The order is given by $(Y,D) \leq (X,C)$ iff $C \subseteq D$ and $Y \restriction C \subseteq X$.
\end{definition}

Consider the map $i_\lambda \colon \tilde{\mathbb{S}}^{*\lambda} \to \mathbb{S}^{*\lambda}$, given by $i_\lambda(X,C) = \bar p$, where for each $\alpha < \lambda$, $p(\alpha)$ is an $\mathbb{S}^{*\alpha}$-name for $$\begin{cases}
    2^{<\omega} &\text{ if } \alpha \notin C\\
    T_{X, \dot{\bar s}^\lambda \restriction (C \cap \alpha)} &\text{ if } \alpha \in C.
\end{cases}$$

\begin{fact}\label{fact:basic}
    The map $i_\lambda$ is a dense embedding. Whenever $\dot x$ is an $\mathbb{S}^{*\lambda}$-name for an element of $\omega^\omega$ and $(X, C) \in \tilde{\mathbb{S}}^{*\lambda}$, there is $(Y,D) \leq (X,C)$ and a continuous function $f \colon Y \to \omega^\omega$ so that $$i_\lambda(Y,D) \Vdash \dot x = f(\dot{\bar s}^\lambda \restriction D).$$
\end{fact}

We may also note then that a sequence $\bar s \in (2^\omega)^\lambda$ is $\mathbb{S}^{*\lambda}$-generic over $V$ exactly if for each dense subset $D \in V$ of $\tilde{\mathbb{S}}^{*\lambda}$, there is $(X, C) \in D$ with $\bar s \restriction C \in X$.

The next lemma is also well-known and not really specific to Sacks forcing. The argument is relatively short and fits well in the line of our other proofs, which is why we include it.

\begin{lemma}\label{lem:realising}
    Let $\dot x$ and $(X,C)$ be as above and let $A \subseteq \omega^\omega$ be either $\mathbf\Sigma^1_1$ or $\mathbf{\Pi}^1_1$. Then there is $(Y,D)$ and $f$ as above, so that additionally, $i_\lambda(Y,D) \Vdash \dot x \in A$ iff $f''Y \subseteq A$. Moreover, for a $\mathbf{\Sigma}^1_2$ set $A$, we may ensure that $f''Y \subseteq A$ if $i_\lambda(Y,D) \Vdash \dot x \in A$.
\end{lemma}

\begin{proof}
If $f''Y \subseteq A$, clearly we must have that $i_\lambda(Y,D) \Vdash \dot x \in A$, as the statement ``$f(\bar y) \in A$, for each $\bar y \in Y$", is absolute. 

Now suppose that $A$ is analytic. More specifically, let $A$ be the projection of a closed set $B \subseteq (\omega^\omega)^2$. For any $(Y, D)$, if $i_\lambda(Y,D) \Vdash \dot x \in A$, there must be a $\mathbb{S}^{*\lambda}$-name $\dot y$ so that $i_\lambda(Y,D) \Vdash (\dot x, \dot y) \in B$. By extending $(Y,D)$ a bit further, using the fact above, we may assume there is a continuous function $g \colon Y \to \omega^\omega$, so that $i_\lambda(Y,D) \Vdash \dot y = g(\dot{\bar s}^\lambda \restriction D)$. Letting $M$ be a countable elementary substructure of some large model $H(\theta)$, note that the subset $Y' \subseteq Y$ with elements of the form $\bar s \restriction D$, for $\bar s$ generic over $M$, is dense in $Y$. For each $\bar y \in Y'$, we have that $(f(\bar y), g(\bar y)) \in B$. But then for each $\bar y \in Y$, $(f(\bar y), g(\bar y)) \in B$, by continuity of $f$ and $g$, and closure of $B$. Followingly, $f''Y \subseteq A$.

Now suppose $A$ is coanalytic. Recall the following representation theorem for $A$ (see e.g. \cite{Kechris}): There is a continuous function $h \colon \omega^\omega \to 2^{\omega}$, so that $A = h^{-1}(\WO)$, where $\WO \subseteq 2^{\omega}$ consists of those $r$, with $E_r = \{ (n,m): r(2^n3^m) = 1 \}$ a well-order on $\omega$. For any ordinal $\alpha < \omega_1$, write $\WO_\alpha$ for the set of $r \in \WO$, where $(\omega, E_r)$ has order-type $\alpha$. If $i_\lambda(Y,D) \Vdash \dot x \in A$, by extending $(Y,D)$, we may assume that $i_\lambda(Y,D) \Vdash h(\dot x) \in \WO_\alpha$, for a particular ordinal $\alpha \in \omega_1^V$, as $\omega_1$ is not collapsed.\footnote{The same representation holds for the reinterpretation of $A$ in any extension of $V$, where we reinterpret $h$ and $\WO$ accordingly.} But then again, the set of $x \in \omega^\omega$, with $h(x) \in \WO_\alpha$, is analytic and a subset of $A$, so we may fall back to the analytic case, which we have already provided.\end{proof}

The following quite simple observation will be used a few times in Section~\ref{sec:reflection}.

\begin{lemma}
    Let $(X, C) \in \tilde{\mathbb{S}}^{*\lambda}$ and $C' \subseteq C$. Then $(X \restriction C', C')$ is a condition.
\end{lemma}

\begin{proof}
    $X \restriction C'$ is closed as it is the projection of a compact set. If $\alpha \in C'$ and $\bar x \in X \restriction (C' \cap \alpha)$, then $T_{X \restriction C', \bar x}$ is clearly the union of perfect trees, as $(X, C)$ was a condition. Thus $T_{X \restriction C', \bar x}$ is itself a perfect tree.
\end{proof}

Also the following is well-known and due to Sacks' original work \cite{Sacks}.

\begin{lemma}\label{lem:singlesacks}
    Let $s$ be a Sacks real over $V$. Then for every real $x \in 2^\omega \cap ( V[s] \setminus V)$, $x$ is a Sacks real over $V$, $V[x] = V[s]$ and there is a continuous injection $f \colon 2^\omega \to 2^\omega$ in $V$, such that $f(x) = s$. 
\end{lemma}

\begin{proof}
    Working in $V$, let $\dot x$ be a name for $x$ and $T \in \mathbb{S}$ arbitrary, such that $T \Vdash \dot x \notin V$. By a standard Axiom A-like fusion argument we find $S \leq T$ and a continuous injection $g \colon [S] \to 2^\omega$, with $S \Vdash g(\dot s) = \dot x$ (see for instance \cite[Lemma 25, 28]{Geschke}). Thinning out $S$ a bit further, if necessary, we can extend $g$ to a total continuous injection $f \colon 2^\omega \to 2^\omega$. A genericity argument thus shows that such $f$ exists with $f(s) = x$.   
    
    Next, $f$ being given, and $T \in \mathbb{S}$ being arbitrary again, $f''[T]$ is an uncountable closed set and thus contains the branches of a perfect tree $S$. If $D \subseteq \mathbb{S}$ is an arbitrary dense set, we find $S' \leq S$ with $S' \in D$. $f^{-1}([S']) \subseteq [T]$ contains the branches of a perfect tree $T'$. Then $T' \leq T$ and $T' \Vdash f(\dot s) \in [S']$, where $S' \in D$. Again, by genericity, $x$ is a generic real for $\mathbb{S}$.
 
    Clearly $V[x] \subseteq V[s]$, as $x \in V[s]$. On the other hand, as $s = f^{-1}(x)$, also $s \in V[x]$, so $V[s] \subseteq V[x]$.
\end{proof}

\section{Reflections}\label{sec:reflection}

\begin{definition}
Let $\lambda$ be an ordinal, $C, D \subseteq \lambda$ countable and $\iota \colon C \to D$ an order-isomorphism. Given a condition $(X,C) \in \tilde{\mathbb{S}}^{*\lambda}$, $\iota$ naturally induces a condition $(\iota[X], D) \in \tilde{\mathbb{S}}^{*\lambda}$, where $$\iota[X] = \{ \langle x_{\iota^{-1}(i)} : i \in D \rangle : \bar x \in X\}.$$ 

Similarly, if $f \colon X \to \omega^\omega$, we define the function $\iota[f]\colon \iota[X] \to \omega^\omega$, where $$\iota[f](\bar x) = f(\langle x_{\iota(i)} : i \in C \rangle ).$$
\end{definition}

We first make a simple observation: 

\begin{lemma}
    Let $(X, C) \in \tilde{\mathbb{S}}^{*\lambda}$, $f \colon X \to \omega^\omega$ be continuous and $\iota \colon C \to D$ be an order-isomorphism. Then $$i_\lambda(X, C) \Vdash f(\dot{\bar s}^\lambda \restriction C) \notin V \text{ iff } i_\lambda(\iota[X], D) \Vdash \iota[f](\dot{\bar s}^\lambda \restriction D) \notin V.$$
\end{lemma}

\begin{proof}
    If $i_\lambda(\iota[X], D)$ does not force $\iota[f](\dot{\bar s}^\lambda \restriction D)$ to be new, then there is some extension $(Z,E) \leq (\iota[X], D)$ so that $f$ is constant on $Z\restriction D$. $(Z \restriction D, D)$ is a condition, $(\iota^{-1}[Z \restriction D], C) \leq (X,C)$ and $f = \iota^{-1}[\iota[f]]$ is clearly also constant on $(\iota^{-1}[Z \restriction D], C)$. Thus $i_\lambda(X, C)$ could not have forced $f(\dot{\bar s}^\lambda \restriction C)$ to be new. Similarly, for the converse implication.
\end{proof}

\begin{lemma}\label{lem:emb12}
    Let $\varphi(v)$ be a $\Pi^1_2(r)$-formula, $r \in (\omega^\omega)^V$. Let $(X, C) \in \tilde{\mathbb{S}}^{*\lambda}$, $f \colon X \to \omega^\omega$ continuous and $\iota \colon C \to D$ an order-isomorphism. Then $$i_\lambda(X, C) \Vdash \varphi(f(\dot{\bar s}^\lambda \restriction C)) \text{ iff } i_\lambda(\iota[X], D) \Vdash \varphi(\iota[f](\dot{\bar s}^\lambda \restriction D)).$$
\end{lemma}

\begin{proof}
    As $X = \iota^{-1}[\iota[X]]$ and $f = \iota^{-1}[\iota[f]]$, we really only need to prove the forward direction of the equivalence above. 
    
    Assume that $i_\lambda(X, C) \Vdash \varphi(f(\dot{\bar s}^\lambda \restriction C))$, but there is an extension $(Z, E) \leq (\iota[X], D)$, such that $i_\lambda(Z, E) \Vdash \neg \varphi(\iota[f](\dot{\bar s}^\lambda \restriction D))$. Let $\varphi(v) = \forall v_0 \psi(v, v_0)$, for a $\Sigma^1_1(r)$-formula $\psi$. By Lemma~\ref{lem:realising}, we can assume that for every $\bar z \in Z$, $\neg \psi(\iota[f](\bar z \restriction C), g(\bar z))$, for some continuous $g \colon Z \to \omega^\omega$. There is a Borel function $h \colon Z \restriction D \to Z$ so that $h(\bar z) \restriction D = \bar z$, for instance by the uniformization theorem for compact sections (see \cite[Theorem 18.18]{Kechris}). Then we have that $(Z \restriction D,D)$ is a condition, $(\iota^{-1}[Z \restriction D], C) \leq (X,C)$ and for every $\bar x \in \iota^{-1}[Z \restriction D]$, $\neg \psi(f(\bar x), \iota^{-1}[g\circ h](\bar x))$. This last statement is absolute, as it can be expressed by a $\mathbf\Pi^1_1$-formula. But then $i_\lambda(\iota^{-1}[Z \restriction D], C)$ forces $\neg \varphi(f(\dot{\bar s}^\lambda \restriction C))$ -- contradiction.
\end{proof}

\begin{thm}[Reflection for $\Sigma^1_3$]\label{thm:ref13}
Let $\lambda$ be any ordinal. Let $\varphi(v)$ be $\Sigma^1_3(r)$, $r \in (\omega^\omega)^V$, and let $\bar s$ be $\mathbb{S}^{*\lambda}$ generic over $V$. Then the following are equivalent: \begin{enumerate}
    \item There is $x \in \omega^\omega \cap (V[\bar s] \setminus V)$, with $V[\bar s] \models \varphi(x)$.
    \item There is $x \in \omega^\omega \cap (V[\bar s \restriction \omega_1] \setminus V)$, with $V[\bar s \restriction \omega_1] \models \varphi(x)$.
\end{enumerate}
\end{thm}

\begin{proof}
    The implication from (2) to (1) follows trivially from the upwards-absoluteness of $\Sigma^1_3$ formulas. Let $\varphi(v) = \exists v_0 \psi(v, v_0)$, where $\psi$ is $\Pi^1_2(r)$. Then (1) implies that there are $(X, C) \in \tilde{\mathbb{S}}^{*\lambda}$, a continuous function $f \colon X \to \omega^\omega$, and continuous $g \colon X \to \omega^\omega$, so that $$i_\lambda(X,C) \Vdash \psi(f(\dot{\bar s}^{\lambda} \restriction C), g(\dot{\bar s}^{\lambda} \restriction C)) \wedge f(\dot{\bar s}^{\lambda} \restriction C) \notin V.$$
    Consider the unique order-isomorphism $\iota \colon C \to \alpha$, where $\alpha < \omega_1$ is the order-type of $C$. According to Lemma~\ref{lem:emb12}, also $$i_{\lambda}(\iota[X],\alpha) \Vdash_{\mathbb{S}^{*\lambda}} \psi(\iota[f](\dot{\bar s}^{\alpha}), \iota[g](\dot{\bar s}^{\alpha})) \wedge \iota[f](\dot{\bar s}^{\alpha}) \notin V.$$

    By Shoenfield's absoluteness theorem, the same is forced relative to $\mathbb{S}^{*\omega_1}$. This finishes the proof: $\mathbb{S}^{*\omega_1}$ is homogeneous, so if any condition forces a new real satisfying $\varphi(v)$, so does the trivial condition. 
\end{proof}

The following is an immediate corollary: 

\begin{cor}[$\Sigma^1_3$ absoluteness]\label{cor:abs13}
    Let $\lambda$ be any ordinal and $\bar s$ be $\mathbb{S}^{*\lambda}$ generic over $V$. Then $\Sigma^1_3$ formulas are absolute between $V[\bar s \restriction \omega_1]$ and $V[\bar s]$.
\end{cor}

\begin{proof}
    $\Sigma^1_3$-formulas are upwards-absolute in general. On the other hand, suppose $x \in V[\bar s \restriction \omega_1]$ and $V[\bar s] \models \varphi(x)$, for a $\Sigma^1_3$-formula $\varphi$. Then there is $\alpha < \omega_1$, so that $x \in V[\bar s \restriction \alpha]$. Consider the tail of the iteration after stage $\alpha$. This is an iteration of length $\lambda'$, where $\lambda = \alpha + \lambda'$, over $V[\bar s \restriction \alpha]$. According to Theorem~\ref{thm:ref13}, replacing $V$ with $V[\bar s \restriction \alpha]$, the $\Sigma^1_3(x)$-sentence $\varphi(x)$ reflects to $V[\bar s \restriction \alpha][\bar s \restriction [\alpha, \alpha + \omega_1)] = V[\bar s \restriction \omega_1]$, as required.
\end{proof}

The reflection theorem above is optimal in general. For instance, if $V = L$ and $\lambda$ is a successor ordinal, say $\lambda = \omega_1 + 1$, then there is a maximal $L$-degree in $V[\bar s]$. The set of reals of maximal $L$-degree is a $\Pi^1_3$ set. But there is no maximal degree in $V[\bar s \restriction \omega_1]$. 

On the other hand, if $\lambda$ has uncountable cofinality -- which is exactly to say that there is no maximal degree in $V[\bar s]$, see Fact~\ref{fact:degrees} -- then we may obtain reflection up to $\Sigma^1_4$, see Theorem~\ref{thm:ref14} below. This is then optimal for any $\lambda > \omega_1$. Namely, if $V = L$, in $V[\bar s]$, there is a non-empty $\Pi^1_4$ set of reals of degree beyond $\omega_1$ -- namely, the set of $x$ so that $$\forall \langle y_n :n \in \omega \rangle \in (2^\omega)^\omega ( \forall n (y_n <_L x) \rightarrow \exists y ( \forall n (y_n <_L y) \wedge y <_L x) ).$$

\begin{lemma}\label{lem:emb13}
    Let $\cf(\lambda) > \omega$ and $\varphi(v)$ be $\Pi^1_3(r)$, $r \in (\omega^\omega)^V$. Let $(X, C) \in \tilde{\mathbb{S}}^{*\lambda}$, $f \colon X \to \omega^\omega$ continuous and $\iota \colon C \to \alpha$ be the collapse of $C$. Then $$i_\lambda(X, C) \Vdash \varphi(f(\dot{\bar s}^\lambda \restriction C)) \text{ implies } i_\lambda(\iota[X], \alpha) \Vdash \varphi(\iota[f](\dot{\bar s}^\alpha)).$$
\end{lemma}

\begin{proof}
    Let $\varphi(v) = \forall v_0 \psi(v,v_0)$, where $\psi$ is $\Sigma^1_2(r)$. Suppose that $i_\lambda(\iota[X], \alpha)$ does not force $\varphi(\iota[f](\dot{\bar s}^\alpha))$. Then there is $(Y, D) \leq (\iota[X], \alpha)$ and a continuous function $g \colon Y \to \omega^\omega$ so that $$i_\lambda(Y,D) \Vdash \neg \psi(\iota[f](\dot{\bar s}^\alpha), g(\dot{\bar s}^{*\lambda} \restriction D)).$$
    Now we can extend $\iota^{-1}$ to an order-isomorphism $\eta \colon D \to E$, for some $E \subseteq \lambda$. By Lemma~\ref{lem:emb12}, $i_\lambda(\eta[Y], E) \Vdash \neg \psi(f(\dot{\bar s}^\lambda \restriction C), g(\dot{\bar s}^{*\lambda} \restriction E))$. As $(\eta[Y], E) \leq (X, C)$, this means that $i_\lambda(X,C)$ could not have forced $\varphi(f(\dot{\bar s}^\lambda \restriction C))$.
\end{proof}

The converse statement of the lemma certainly is false. For example, forcing over $L$, consider the $\Pi^1_3$-formula $\varphi(x,y) = \forall z ( z \leq_L x \vee y \leq_L z)$. The trivial condition forces $\varphi(\dot s_0, \dot s_1)$ and $\neg \varphi(\dot s_0, \dot s_2)$.

\begin{thm}[Reflection for $\Sigma^1_4$]\label{thm:ref14}
Let $\cf(\lambda) > \omega$ and $\varphi(v)$ be $\Sigma^1_4(r)$, for $r \in (\omega^\omega)^V$. Let $\bar s$ be $\mathbb{S}^{*\lambda}$ generic over $V$. Then the following are equivalent: \begin{enumerate}
    \item There is $x \in \omega^\omega \cap (V[\bar s] \setminus V)$, with $V[\bar s] \models \varphi(x)$.
    \item There is $x \in \omega^\omega \cap (V[\bar s \restriction \omega_1] \setminus V)$, with $V[\bar s \restriction \omega_1] \models \varphi(x)$.
\end{enumerate}
\end{thm}

\begin{proof}
   For the implication from (2) to (1), simply note that Corollary~\ref{cor:abs13} implies the upwards-absoluteness of $\Sigma^1_4$-formulas from $V[\bar s \restriction \omega_1]$ to $V[\bar s]$.

    For the other direction, we argue as in Theorem~\ref{thm:ref13}. Let $\varphi(v) = \exists v_0 \psi(v,v_0)$, for a $\Pi^1_3(r)$-formula $\psi$.  We have $(X, C) \in \tilde{\mathbb{S}}^{*\lambda}$, a continuous function $f \colon X \to \omega^\omega$, and continuous $g \colon X \to \omega^\omega$, so that $$i_\lambda(X,C) \Vdash_{\mathbb{S}^{*\lambda}} \psi(f(\dot{\bar s}^{\lambda} \restriction C), g(\dot{\bar s}^{\lambda} \restriction C)) \wedge f(\dot{\bar s}^{\lambda} \restriction C) \notin V.$$ Given the collapsing function $\iota \colon C \to \alpha$, where $\alpha$ is the order-type of $C$, $$i_{\lambda}(\iota[X],\alpha) \Vdash_{\mathbb{S}^{*\lambda}} \psi(\iota[f](\dot{\bar s}^{\alpha}), \iota[g](\dot{\bar s}^{\alpha})) \wedge \iota[f](\dot{\bar s}^{\alpha})\notin V,$$ by Lemma~\ref{lem:emb13}. It follows from the downwards absoluteness of $\Pi^1_3$-formulas that the same is forced relative to $\mathbb{S}^{*\omega_1}$. Again, this finishes the proof.    
\end{proof}

Just as before, we obtain: 

\begin{cor}[$\Sigma^1_4$ absoluteness]\label{cor:abs14}
    Let $\lambda$ be any ordinal and $\bar s$ be $\mathbb{S}^{*\lambda}$ generic over $V$. Then $\Sigma^1_4$ formulas are absolute between $V[\bar s \restriction \omega_1]$ and $V[\bar s]$.
\end{cor}

\begin{remark}
    It is conceivable that $\omega_1$ above can be replaced by $\delta^1_3(r)$ and $\delta^1_4(r)$ respectively, although we have not put much thought into it. Recall that $\delta^1_n(r)$ is the least ordinal not order-isomorphic to a $\Delta^1_n(r)$-definable well order on $\omega$.
\end{remark}

\section{Coding digression}\label{sec:coding}

For $z_0, z_1 \in 2^\omega$, we let $z_0 \oplus z_1 = z \in 2^\omega$, where $z(2n) = z_0(n)$ and $z(2n+1) = z_1(n)$, for $n \in \omega$. Recall that $\WO_\alpha = \{ r \in 2^\omega : (\omega, E_r) \cong (\alpha, \in) \}$, where $E_r = \{ (n,m): r(2^n3^m) = 1 \}$. Fix an effective enumeration $\langle s(n) : n \in \omega \rangle$ of $\omega^{<\omega}$ and an effective bijection $e \colon (2^{<\omega})^{<\omega} \to \omega$.

\begin{definition}
    A real $u \in \omega^\omega$ codes the continuous function $f \colon (2^\omega)^\omega \to \omega^{<\omega}$, if for every $\bar x \in (2^\omega)^\omega$, $$f(\bar x) = \bigcup_{n \in \omega} s\Bigl(u\bigl(e(\langle x_i \restriction n) : i < n \rangle)\bigr)\Bigr).$$ 
    
    A real $z \in 2^\omega$ is a code for the condition $(X, \alpha) \in \tilde{\mathbb{S}}^{*\omega_1}$, if $ z = z_0 \oplus z_1$, where $z_0 \in \WO_\alpha$ and, letting $\iota \colon \omega \to \alpha$ be the unique order-isomorphism between $(\omega, E_{z_0})$ and $(\alpha, \in)$, $$X = \{ \bar x \in (2^\omega)^\alpha : \forall n,m \in \omega (z_1(e(\langle x_{\iota(i)} \restriction m : i < n \rangle)) = 1 ) \}.$$

    Given a code $u$ for a continuous function $f$ as above, we write $f_{z,u}$ for the function $X \to \omega^\omega$, given by $$f_{z, u}(\bar x) = f(\langle x_{\iota(n)} : n \in \omega \rangle).$$

    We also write $(X_z, \alpha_z)$ for the unique condition coded by $z \in 2^\omega$. 
\end{definition}

Clearly every condition $(X, \alpha)$ has a code $z$, and whenever $f \colon X \to \omega^\omega$ is arbitrary continuous, there is $u$ such that $f = f_{z, u}$. The following is straightforward to check:

\begin{lemma}
    The set of codes for conditions is a $\Pi^1_1$ subset of $2^\omega$. Similarly, the set of codes for continuous function is a $\Pi^1_1$ subset of $\omega^\omega$. The set of pairs of codes for conditions $(z, z')$, where $(X_z, \alpha_z) \leq (X_{z'}, \alpha_{z'})$, is also a $\Pi^1_1$ subset of $(2^\omega)^2$.
\end{lemma}

\begin{lemma}\label{lem:pi12force}
 Let $\varphi(v_0, v_1)$ be $\Pi^1_2(r)$ for some $r \in \omega^\omega$. Then there is another $\Pi^1_2(r)$-formula $\psi(z,u,x)$ expressing that $z$, $u$ are codes and $i_{\omega_1}(X_z, \alpha_z) \Vdash_{\mathbb{S}^{*\omega_1}} \varphi(\check{x}, f_{z, u}( \dot{\bar s}^{\alpha_z}) )$.
\end{lemma}

\begin{proof}
    Let $\varphi(v_0, v_1) = \forall v_2 \psi(v_0,v_1,v_2)$, where $\psi$ is $\Sigma^1_1(r)$. According to Lemma~\ref{lem:realising}, a condition $i_{\omega_1}(X, \alpha) \in \mathbb{S}^{*\omega_1}$ does not force $\varphi(\check{x},  f(\dot{\bar s}^{\alpha} ))$ if and only if there is an extension $(Y, \alpha) \leq (X, \alpha)$ and a continuous function $g \colon Y \to \omega^\omega$ so that $\forall \bar s \in Y (\neg \psi(x, g(\bar s), f(\bar s)))$.\footnote{Note that by Shoenfield's absoluteness theorem, we may argue within $\mathbb{S}^{*\alpha}$ rather than $\mathbb{S}^{*\omega_1}$, which is why we do not need to consider $(Y, \beta) \leq (X, \alpha)$, where $\beta > \alpha$.} Using the previous lemma, for $(X, \alpha) = (X_z, \alpha_z)$ and $f = f_{z, u}$, this can easily be expressed by a $\Sigma^1_2(z,u,x,r)$ formula. 
\end{proof}

\section{Proof of the main result}\label{sec:main}

The proof of Theorem~\ref{thm:main} is quite easy when $\lambda = 1$ and we actually obtain $\MS(\OD, V)$.

\begin{thm}
    Let $s$ be Sacks generic over $V$. Then $V[s] \models \MS(\OD, V)$.
\end{thm}

\begin{proof}
    Let $A$ be $\OD(r)$, $r \in V[s]$. The only interesting case is when $r \in V$ and there is $x \in A \cap (V[s] \setminus V)$. By Lemma~\ref{lem:singlesacks}, $x$ is also a Sacks real over $V$ and $V[x] = V[s]$. Thus is a condition $T$, with $x \in [T]$ and $T \Vdash \varphi(\dot s)$, where $\varphi$ is the $\OD(r)$-formula defining $A$. By Lemma~\ref{lem:singlesacks}, for every $y \in [T] \cap (V[s] \setminus V)$, $V[s] = V[y]$ and $V[s] \models \varphi(y)$. We easily find a perfect subset of $[T] \cap (V[s] \setminus V)$ in $V[s]$. Namely, fixing a homeomorphism $\phi \colon [T] \to 2^\omega$ in $V$, we let $P = \phi^{-1}[\{ s \oplus z : z \in 2^{\omega}\}]$.
\end{proof}

Next, we make the following simple observation. 

\begin{lemma}\label{fact:idealized}
Let $\alpha < \omega_1$, $M \preceq H(\theta)$ countable for some large $\theta$ and $(X, \alpha) \in \tilde{\mathbb{S}}^{*\alpha} \cap M$. Then there is $(Y, \alpha) \leq (X, \alpha)$ so that every $\bar y \in Y$ is $\mathbb{S}^{*\alpha}$-generic over $M$.
\end{lemma}

\begin{proof}
    Let $(Y_0, \alpha) \leq (X,\alpha)$ be a master condition over $M$ and note that the set $B$ of $\mathbb{S}^{*\alpha}$-generics over $M$ is a Borel subset of $(2^\omega)^\alpha$. By Lemma~\ref{lem:realising}, there is $(Y,\alpha) \leq (Y_0, \alpha)$ so that $Y \subseteq B$.
\end{proof}

The crucial part of the proof of Theorem~\ref{thm:main} will be the existence of perfect sets of generics, given in Proposition~\ref{prop:psofgen} below. This is based on the following slightly unusual result which shows that generics can be approximated through initial segments.

\begin{prop}\label{prop:genericity}
    Let $\alpha < \omega_1$ be a limit ordinal, $\bar s$ be $\Ss^{*\alpha}$-generic over $V$ and $\bar t = \langle t_\beta : \beta < \alpha \rangle  \in V[\bar s]$, so that for each $\beta < \alpha$, $\bar t \restriction \beta$ is $\Ss^{*\beta}$-generic over $V$. Then $\bar t$ is also $\Ss^{*\alpha}$-generic over $V$.
\end{prop}

Let's note briefly that the conclusion easily fails when $\bar t$ is not a member of an $\alpha$-length iteration. For instance, in an $\mathbb{S}^{*\omega+1}$ extension one can easily come up with a sequence $\langle t_n : n \in \omega \rangle$, each initial segment being generic for the appropriate length iteration, but coding the degree at level $\omega+1$. 

\begin{proof}
It suffices to show that for each dense subset $D \subseteq \tilde{\mathbb{S}}^{*\alpha}$ in $V$, there is $(X,\alpha) \in D$ with $\bar t \in X$. So fix such $D \in V$ now.

    First note, from the degree structure of $V[\bar s]$ (see Fact~\ref{fact:degrees}), that for each $\beta < \alpha$, $[\bar t \restriction \beta]_V = [\bar s \restriction \beta]_V$ and $[t_\beta]_V = [s_\beta]_V$. For each $\beta$, let $\dot t_\beta$ be an $\mathbb{S}^{*\alpha}$-name for $t_\beta$ and let $\bar p$ be an arbitrary condition forcing this.

    \begin{claim}\label{claim:c1}
        There is $\bar q \leq \bar p$ and for each $\beta < \alpha$, an $\mathbb{S}^{*\beta}$-name $\dot f_\beta$ for a continuous injection $2^\omega \to 2^\omega$, so that $\bar q \Vdash \dot t_\beta = \dot f_\beta(\dot s_\beta)$.
    \end{claim}

    Note in particular, that we may assume that $\dot t_\beta$ is an $\mathbb{S}^{*\beta +1}$-name, rather than just an $\mathbb{S}^{*\alpha}$-name.

    \begin{proof}
        We prove the following statement by induction on $\gamma \leq \alpha$:  
        
        For each $\delta < \gamma$, there is $\bar q \leq \bar p$ and for each $\beta \in [\delta, \gamma)$, an $\mathbb{S}^{*\beta}$-name $\dot f_\beta$, so that \begin{enumerate}
            \item  $\bar q \restriction \delta = \bar p \restriction \delta$, 
            \item for each $\beta  \in [\delta, \gamma)$, $\bar q \Vdash \dot t_\beta = f_\beta(\dot s_\beta)$.
        \end{enumerate}
        
        Suppose this is true for $\gamma < \alpha$. We show that this also holds at $\gamma+1$. Let $\delta < \gamma +1$ be arbitrary. First let $\bar q_0 \leq \bar p$ with $\bar q_0 \restriction \delta = \bar p \restriction \delta$ and $\bar q_0 \Vdash \dot t_\beta = \dot f_\beta(\dot s_\beta)$, $\beta \in [\delta, \gamma)$, using the inductive claim at $\gamma$. By Lemma~\ref{lem:singlesacks}, in the generic extension $V[\bar s]$, there will be a function $f_{\gamma} \in V[\bar s \restriction \gamma]$ such that $t_{\gamma} = f_{\gamma}(s_{\gamma})$. So in $V[\bar s \restriction \gamma]$, there is a condition in the tail of the iteration deciding $f_\gamma$. Back in $V$, we thus find a condition $\bar q \leq \bar q_0$ with $\bar q \restriction \gamma = \bar q_0 \restriction \gamma$ and an $\mathbb{S}^{*\gamma}$-name $\dot f_\gamma$ for the function that is being decided in the tail. Then $\bar q$ and $\dot f_\beta$, $\beta \in [\delta, \gamma +1 )$, are as required.

        Now suppose that this is true for all $\gamma' < \gamma$ and $\gamma \leq \alpha$ is a limit. Let $\langle \delta_n : n \in \omega \rangle$ be strictly increasing and cofinal in $\gamma$, with $\delta_0 = 0$. Using the inductive claim we can easily construct a decreasing sequence $\langle \bar q_n : n \in \omega \rangle$ below $\bar p$, so that for each $n$, \begin{enumerate}
            \item $\bar q_{n+1} \restriction \delta_{n+1} = \bar q_n \restriction \delta_{n+1}$, 
            \item for each $\beta < \delta_{n+1}$, there is an $\mathbb{S}^{*\beta}$-name $\dot f_\beta$, so that $\bar q_n \Vdash \dot t_\beta = \dot f_\beta(\dot s_\beta)$.
        \end{enumerate}
Then $\bar q = \bigcup_{n \in \omega} \bar q_n \restriction [\delta_n, \delta_{n+1})$ clearly works.
    \end{proof}

Fix now $\bar q \leq \bar p$ and $\dot f_\beta$, $\beta < \alpha$, as in the claim. Further, let $(X_0,\alpha) \in \tilde{\mathbb{S}}^{*\alpha}$ and $f \colon X_0 \to (2^\omega)^\alpha$ continuous, such that $i_\alpha(X_0,\alpha)$ extends $\bar q$ and forces that $f(\bar s) = \bar t$. Then $i_\alpha(X_0,\alpha) \Vdash f(\dot{\bar s}^{\alpha})(\check{\beta}) = \dot f_\beta (\dot s_\beta)$, for each $\beta < \alpha$.

Next, let $M$ be a countable elementary submodel of some large $H(\theta)$ containing all relevant parameters. Let $(X, \alpha) \leq (X_0, \alpha)$ so that every $\bar x \in X$ is $\mathbb{S}^{*\alpha}$-generic over $M$ (see Lemma~\ref{fact:idealized}). Note that for each $\bar x \in X$ and $\beta < \alpha$, $f(\bar x)\restriction \beta$ only depends on $\bar x \restriction \beta$, as each value $f(\bar x)(\delta)$, $\delta < \beta$, is the evaluation of the $\mathbb{S}^{*\delta +1}$-name $\dot t_\delta \in M$ according to the $\mathbb{S}^{\delta+1}$ generic over $M$ given by $\bar x \restriction (\delta +1)$ (see the sentence after Claim~\ref{claim:c1}). Let $g_\beta \colon X \restriction \beta \to (2^{\omega})^\beta$ be the function where $g_\beta(\bar x \restriction \beta) = f(\bar x) \restriction \beta$.

\begin{claim}
    Each $g_\beta$ is injective. 
\end{claim}

\begin{proof}
    Suppose not and let $\beta$ be least so that $g_\beta$ is not injective. Clearly, $\beta$ must be a successor ordinal, say $\beta = \delta +1$. Let $\bar x, \bar x' \in X \restriction \beta$ with $g_\beta(\bar x) = g_\beta(\bar x')$. As $\beta$ is least, $\bar x \restriction \delta = \bar x' \restriction \delta$. Then let $f_\delta \in M[\bar x \restriction \delta] = M[\bar x' \restriction \delta]$ be the interpretation of $\dot f_\delta \in M$ according to the $\mathbb{S}^{*\delta}$-generic given by $\bar x \restriction \delta$. But then $g_\beta(\bar x)(\delta) = f_\delta(x_\delta) = f_\beta(x'_\delta) = g_\beta(\bar x')(\delta)$, contradicting the injectivity of $f_\delta$. 
\end{proof}

\begin{claim}
    $(f''X, \alpha)$ is a condition.\footnote{For a similar claim, see \cite[Lemma 7]{Kanovei1999}}
\end{claim}

\begin{proof} 
    Clearly $f''X$ is closed. If $\beta < \alpha$ and $\bar y \in f''X \restriction \beta$, then $\bar y = g_\beta(\bar x)$ for a unique $\bar x \in X \restriction \beta$. Similarly to before, let $f_\beta$ be the interpretation of $\dot f_\beta$ according to $\bar x$. We obtain that $[T_{f'' X, \bar y}] = {f_\beta}''[T_{X, \bar x}]$. As $f_\beta$ is a continuous injection and $[T_{X, \bar x}]$ is a perfect set, so is $[T_{f'' X, \bar y}]$. 
\end{proof}

Now find $(Y, \alpha) \leq (f''X,\alpha)$ in $D$. 

\begin{claim}
    $(f^{-1}(Y), \alpha) \leq (X, \alpha)$ is a condition. 
\end{claim}

\begin{proof}
    The argument is exactly as before. Let $\bar x \in f^{-1}(Y) \restriction \beta$ and $g_\beta(\bar x) = \bar y \in Y \restriction \beta$. If $f_\beta$ is the interpretation of $\dot f_\beta$ according to $\bar x$, then, as before, $[T_{f^{-1}(Y), \bar x}] = {f_\beta}^{-1}[T_{Y, \bar y}]$.
\end{proof}

Finally, we have that $i_\alpha(f^{-1}(Y), \alpha)$ extends $\bar p$ and forces that $\bar t \in Y$, where $(Y,\alpha) \in D$. As $\bar p$ was arbitrary, this finishes the proof.
\end{proof}

\begin{prop}\label{prop:psofgen}
    Let $\alpha < \omega_1$, $\bar s$ be $\Ss^{*\alpha}$-generic over $V$ and $(X, \alpha) \in \tilde{\mathbb{S}}^{*\alpha}$. Then in $V[\bar s]$ there is a perfect set of $\Ss^{*\alpha}$ generics $\bar t \in X$. 
\end{prop}

\begin{proof}
The statement is vacuous or simply uninteresting when $\alpha = 0$. Also, as the forcing is homogeneous, we may assume without loss of generality that $\bar s \in X$.

For $\alpha = \beta +1$, consider the perfect tree $T = T_{X, \bar s \restriction \beta} \in V[\bar s \restriction \beta]$. Then note that in $V[\bar s] = V[\bar s \restriction (\beta +1)]$, there is a perfect set $P_0$ of reals $x \in [T] \setminus V[\bar s \restriction \beta]$.\footnote{Namely, if $f \colon [T] \to 2^\omega$ is a continuous bijection coded in $V[\bar s \restriction \beta]$, then $f^{-1}(\{s_\beta \oplus z : z \in 2^\omega\})$ is such a set.} Recall, by Lemma~\ref{lem:singlesacks}, that every new real in $V[\bar s]$ is generic for $\mathbb{S}$ over $V[\bar s \restriction \beta]$. Thus $P = \{ \bar s \restriction \beta^\frown x : x \in P_0 \} \subseteq X$ is a perfect set of generics over $V$.
    
    For $\alpha$ limit, fix a cofinal sequence $\langle \alpha_n : n \in \omega \rangle \in V$ in $\alpha$. In $V[\bar s]$, find a map $\sigma$ with domain $2^{<\omega}$ so that \begin{itemize}
        \item for each $n \in \omega$ and $\eta \in 2^n$, $\sigma(\eta) \in X \restriction \alpha_n$ is $\Ss^{*\alpha_n}$-generic over $V$, 
        \item for each $\eta \neq \nu$ in $2^{<\omega}$, $\sigma(\eta) \neq \sigma(\nu)$, 
        \item and for each $\eta \subseteq \nu$ in $2^{<\omega}$, $\sigma(\eta) \subseteq \sigma(\nu)$.
    \end{itemize}  For $x \in 2^\omega$, let us write $\sigma(x)$ for $\bigcup_{n \in \omega} \sigma(x \restriction n)$. The set $P = \{\sigma(x) : x \in 2^\omega \}$ is obviously a perfect subset of $X$ and by Proposition~\ref{prop:genericity} every member of $P$ in $V[\bar s]$ is generic for $\mathbb{S}^{*\alpha}$.
\end{proof}

\begin{lemma}\label{lem:injection}
    Let $\alpha < \omega_1$, $\bar p \in \mathbb{S}^{*\alpha}$ and $\dot x$ be a $\mathbb{S}^{*\alpha}$-name for an element of $\omega^\omega$ so that $\bar p \Vdash [\dot x]_V = [\dot{\bar s}^{\alpha}]_V$. Then there is $(X, \alpha)$ and a continuous injection $f \colon X \to \omega^\omega$ so that $i_\alpha(X, \alpha) \leq \bar p$ and $i_\alpha(X, \alpha) \Vdash \dot x = f(\dot{\bar s}^{\alpha})$.
\end{lemma}

\begin{proof}
    Since $\bar p$ forces that $\dot x$ and $\dot{\bar s}^{\alpha}$ have the same $V$-degree, there is an extension $\bar q \leq \bar p$ forcing that $\chi$ uniquely defines $\dot{\bar s}^{\alpha}$ from $\dot x$, for some formula $\chi$ with parameters in $V$. We may also assume that there is a continuous function $f$, with $\bar q \Vdash f(\dot{\bar s}^{\alpha}) = \dot x$. 
    
    Consider an elementary submodel $M$ containing all relevant information. Let $i_\alpha(X,\alpha) \leq \bar q$ so that every $\bar y \in X$ is generic over $M$. Then we claim that $f$ must be injective on $X$: If $f(\bar y) = f(\bar y') = : x$, then $M[\bar y] = M[x] = M[\bar y']$. But then both $\bar y$ and $\bar y'$ are defined, within the same model, by the formula $\chi$ from $x$, so $\bar y = \bar y'$.
\end{proof}

\begin{proof}[Proof of Theorem~\ref{thm:main} for $\cf(\lambda) > \omega$]
    Let $\bar s$ be generic and $A = \{ x \in \omega^\omega : \exists y\varphi(x,y) \}$, where $\varphi$ is $\Pi^1_2(r)$. By working relative to $V[r]$ and noting that any tail of the iteration still has length of uncountable cofinality, we may assume without loss of generality that $r \in V$.
    
    Let $x \in A \setminus V$. By Thm~\ref{thm:ref13} and Fact~\ref{fact:degrees}, we have that $[x]_V = [\bar s \restriction \alpha]_V$, for some $\alpha < \omega_1$, and that $V[\bar s \restriction \omega_1] \models \exists y \varphi(x, y)$. Then in $V[x] = V[\bar s \restriction \alpha]$, a condition relative to $\mathbb{S}^{*\omega_1}$ forces $\exists y\varphi(\check{x},y)$. This means that there is $(X, \beta) \in \tilde{\mathbb{S}}^{*\omega_1}$ and a continuous function $f \colon X \to \omega^\omega$, so that $i_{\omega_1}(X,\beta) \Vdash \varphi(\check x, f(\dot{\bar s}^\beta))$.   According to Lemma~\ref{lem:pi12force}, there is a $\Pi^1_2(r)$ formula $\psi(z,u,x)$ which expresses this forcing statement, where $z$ and $u$ are codes for $(X,\beta)$ and $f$. 

    Back in $V$, there is a condition $(X,\alpha) \in \tilde{\mathbb{S}}^{*\omega_1}$, a continuous injection $g \colon X \to \omega^\omega$ (see Lemma~\ref{lem:injection}) and continuous functions $g_0, g_1$, so that $$i_{\alpha}(X,\alpha) \Vdash_{\mathbb{S}^{*\alpha}} \psi(g_0(\dot{\bar s}^\alpha), g_1(\dot{\bar s}^\alpha), g(\dot{\bar s}^\alpha)).$$ By Proposition~\ref{prop:psofgen}, in $V[\bar s \restriction \alpha]$ there is a perfect set $P \subseteq X$ of $\mathbb{S}^{*\alpha}$-generics over $V$. In particular, in that model, for any $\bar t \in P$, $\psi(g_0(\bar t), g_1(\bar t), g(\bar t))$. By Shoenfield-absolutness, even for any $\bar t \in P^{V[\bar s]}$, $\psi(g_0(\bar t), g_1(\bar t), g(\bar t))$. 
    
    Work in $V[\bar s]$ now. As $g$ is injective, $g''P$ is a perfect subset of $\omega^\omega$. Let $x' \in g''P$ be arbitrary, say $x' \in V[\bar s \restriction \gamma]$, for some $\gamma < \lambda$. Then, relative to $V[\bar s \restriction \gamma]$, there is a condition in $\tilde{\mathbb{S}}^{*\omega_1}$ forcing $\exists y\varphi(\check{x}',y)$. Since a generic over $V[\bar s \restriction \gamma]$ containing such condition exists and applying Shoenfield-absoluteness one more time, we find that $V[\bar s] \models \exists y\varphi(x',y)$. All together, we have shown that $g''P \subseteq A$. 
\end{proof}

The proof also shows the following:

\begin{thm}
    Let $\lambda$ be an ordinal of uncountable cofinality. Let $\bar s$ be $\mathbb{S}^{*\lambda}$-generic over $V$. In $V[\bar s]$, suppose that $A$ is $\Sigma^1_3(r)$ and $x \in A \setminus V[r]$. Then $A$ contains reals in every degree above $x$.
\end{thm}

\begin{proof}
    The proof of Theorem~\ref{thm:main} produces a perfect subset $[T] \subseteq A$ with $T \in V[x]$. For any $y \geq_V x$, there is a branch $y'$ of $T$ in $V[y]$ so that $y \leq_V y' \oplus T$. If $y' <_V y$, then also $y' \oplus x = \max_{\leq_V}(y', x) <_V y$. But then $y \leq_V y' \oplus T \leq_V y' \oplus x <_V y$, posing a contradiction.
\end{proof}

The minimal degree in a $\Pi^1_2$-set can certainly appear after the first Sacks real. For instance, forcing over $L$, consider the $\Pi^1_2$ formula expressing that $x$ is of the form $x_0 \oplus x_1$ and $x_1 \not\leq_L x_0 \not\leq_L 0$.

\begin{proof}[Proof of Theorem~\ref{thm:mainL}]
    Let $\bar s = \langle s_\alpha : \alpha < \lambda \rangle$ be $\mathbb{S}^{*\lambda}$-generic over $L$. We already know that $L[\bar s] \models \MS(\Sigma^1_3, L)$, when $\lambda = 1$ or $\lambda$ is a limit of uncountable cofinality.

 Only assuming $\lambda > 1$, $\MS(\Pi^1_3)$ fails because the set $A$ of minimal reals over $L$ is $\Pi^1_3$ but contains no perfect subset. Namely, suppose that $P \in L[\bar s]$ is a perfect subset of $A$. Then $L[\bar s]\models \forall x \in P ( s_1 \not\leq_L x)$. Add a single additional Sacks real $s_\lambda$ over $V[\bar s]$. In $L[\bar s][s_\lambda]$, there is a new real $x$ in $P$ of degree $s_\lambda$ and by the degree structure, $s_1 \leq_L x$. On the other hand, by Shoenfield absoluteness, $s_1 \not\leq_L x$ should hold -- contradiction.\footnote{Another way to see that $A$ has no perfect subset is through the Groszek-Slaman Theorem, which shows that if $2^\omega \setminus L[s_0] \neq \emptyset$, then no perfect set can be contained in $L[s_0]$, see \cite[Theorem 2.4]{Groszek1998}}.

    If $\lambda > 1$ is not a limit of uncountable cofinality, then there is a maximal degree $[a]_L$ in $V[\bar s]$. The set $A$ of reals of non-maximal degree is $\Sigma^1_3$ and not contained in $L$. Suppose that there is a perfect set $P\subseteq A$ in $V[\bar s]$. As before, we obtain a contradiction from $\forall x \in P ( a \not\leq_L x)$.
\end{proof}

\begin{remark}
    The classical Mansfield-Solovay Theorem can be strengthened to say that a $\Sigma^1_2(r)$ set $A$ is either contained in $L[r]$ or contains a perfect set which is additionally \emph{coded in $L[r]$}. Clearly, we cannot ever hope to generalize this even to $\Pi^1_2$ sets, as witnessed by the set of non-constructible reals. Nevertheless, as mentioned above, our argument shows that (when forcing over $L$) the perfect set can be found in $L[r,x]$, where $x$ is any member of $A$. Of course, Shoenfield's absolutness theorem shows that a non-empty $\Sigma^1_2(r)$ set already contains an element in $L[r]$ itself.
\end{remark}

Finally, let us repeat the argument for Marczeswki measurability from \cite{BrendleLoewe}.

\begin{proof}[Proof of Theorem~\ref{thm:marcz}]
    In the forcing extension, let $A$ be ${\Sigma}^1_3(r)$ and $P$ a perfect set. The set $P \cap A$ is a ${\Sigma}^1_3(r, x)$ set, where $x$ is a code for $P$. If $P\cap A \subseteq V[r,x]$ simply let $P' \subseteq P$ be a perfect set disjoint from $V[r,x]$, which can easily be shown to exist. The other option is that $P \cap A$ outright contains a perfect subset $P'$.
\end{proof}

\section{\texorpdfstring{$\MS(\Sigma^1_n)$ vs $\MS(\Pi^1_n)$}{MS(Σ1n) vs MS(Π1n)}}\label{sec:Kanovei}

We now explain how one can separate the Mansfield-Solovay Theorem at the different levels of the projective hierarchy. We base ourselves on the more modern exposition of the ideas in \cite{Kanovei1978} given in \cite{Kanovei2024}. In \cite{Kanovei2024}, for each $n \geq 3$, a forcing notion $\mathbb{P}(n)$ is introduced that adds a generic real $a \in 2^\omega$ over $L$ that is $\Delta^1_{n+1}$ definable, and where $\{ a\}$ is a $\Pi^1_n$ singleton (see the paragraph right before Section 11), while every $\Sigma^1_n$ definable real is constructible. We will not define the forcing in full detail here, as it is beyond the scope of our paper and we only need a few key properties that can be found explicitly in \cite{Kanovei2024}. Nevertheless, to the reader's convenience, let us give a rough sketch.

Work in $L$. Just like Jensen's forcing notion, $\mathbb{P}(n)$ consists of (a certain collection of) perfect subtrees of $2^{<\omega}$ ordered under inclusion. First, a tree $\mathbf{J}_{<\omega_1}$ of so called \emph{Jensen sequences} is defined, whereby each element of $\mathbf{J}_{<\omega_1}$ is a sequence $\langle P_\xi : \xi < \alpha \rangle$, $\alpha < \omega_1$, of countable sets $P_\xi$ of perfect trees. Every branch $\langle P_\xi : \xi < \omega_1 \rangle$ defines a forcing notion $\mathbb{P} = \bigcup_{\xi < \alpha} P_\xi$. The notion of a Jensen sequence captures the essential properties of Jensen's original recursive construction. The following is shown outright for each forcing notion $\mathbb{P}$ obtained in such a way.

\begin{lemma}[{\cite[Lemma 9, 10, 11]{Kanovei2024}}]\label{lem:K1}
    Let $G$ be $\mathbb{P}$-generic over $L$. Then $\bigcap_{T \in G} [T]$ is a singleton $\{ a\}$, so that $L[G] = L[a]$ and $a$ is minimal, i.e. $[0]_L$ and $[a]_L$ are the only degrees in $L[G]$. Moreover, for any real $x \in 2^\omega \cap L[a]$, there is a continuous function $f \colon 2^\omega \to 2^\omega$ coded in $L$, so that $f(a) = x$.
\end{lemma}

For each $n$, $\mathbb{P}(n)$ is obtained from a branch $\bar P$ through $\mathbf{J}_{<\omega_1}$ that is $\Delta_{n-1}$ definable over $H(\omega_1)$ and such that $\bar P$ \emph{solves} every $\mathbf{\Sigma}_{n-2}$ definable $D \subseteq \mathbf{J}_{<\omega_1}$: there is $\alpha < \omega_1$ so that, either $\bar P \restriction \alpha \in D$, or no extension of $\bar P \restriction \alpha$ in $\mathbf{J}_{<\omega_1}$ is in $D$.

Let $\dot a$ be a name for the generic real of $\mathbb{P}(n)$. The crucial property is the following: 

\begin{lemma}[{\cite[Lemma 16, Theorem 5]{Kanovei2024}}]\label{lem:K2}
    Let $G$ be $\mathbb{P}(n)$-generic over $L$ and $\varphi$ a $\mathbf{\Pi}^1_{n-1}$-formula with parameters in $L$. Then $V[G] \models \varphi(\dot a^G)$ if and only if there is a condition $T \in G$, so that $T \Vdash_{\mathbb{S}} \varphi(\dot s)$.
\end{lemma}

\begin{proof}[Proof of Theorem~\ref{thm:kanovei}]
   We already know that $\MS(\Pi^1_2)$ can fail consistently. So let $n\geq 3$. Let $G$ be $\mathbb{P}(n)$-generic over $L$. Since there is a $\Pi^1_n$ singleton, $L[G] \models \neg \MS(\Pi^1_n)$. Now let $A = \{x \in \omega^{\omega}: \exists y\varphi(x,y) \}$, where $\varphi$ is $\Pi^1_{n-1}(r)$, $r \in L[G]$. By the minimality of the extension, the only relevant case is when $r \in L$ and there is $x \in A \setminus L$. Then there is $y \in \omega^\omega$ so that $L[G] \models \varphi(x,y)$. By Lemma~\ref{lem:K1}, there are continuous functions $f, g\in L$ such that $x = f(\dot a^G)$ and $y = g(\dot a^G)$. According to Lemma~\ref{lem:K2} above, there is $T \in G$ so that $T \Vdash_{\mathbb{S}} \varphi(f(\dot s), g(\dot s))$.\footnote{$\varphi(f(\dot s), g(\dot s))$ can of course be rewritten as a $\mathbf{\Pi}^1_{n-1}$ property of only $\dot s$, using codes for $f$ and $g$ in $L$.} Recall that in a Sacks forcing extension, every new element of $[T]$ is a Sacks real with $T$ in its generic filter. There is a Borel function $h$ coded in $L$ (in fact even continuous), so that whenever $s$ is Sacks generic over $L$, $h(s)$ is a perfect subtree of $T$ with $[h(s)] \cap L = \emptyset$. Just fix a homeomorphism $\phi \colon [T] \to 2^\omega$ in $L$, andar let $h(s)$ be such that $[h(s)] = \phi^{-1}[\{ s \oplus z : z \in 2^{\omega}\}]$.
   
   Thus, we find that $$T \Vdash_{\mathbb{S}} \forall x \in [h(\dot s)] \varphi(x,g \circ f^{-1}(x)).$$
   The forced statement can be expressed by a $\mathbf{\Pi}^1_{n-1}$-formula using parameters in $L$. So applying Lemma~\ref{lem:K2} once again, we find that $[h(\dot a^G)]$ is a perfect subset of $A$ in $L[G]$.
\end{proof}

Of course, it would also be interesting to separate these in models not of the form $L[a]$ for a real $a$.

\section{Open questions}

Recall that the side-by-side Sacks model is obtained using a countable support product of Sacks forcing. The degree structure of this model is more difficult to describe than that of the iteration, and it has recently been shown that it is rigid (see \cite{Notaro}).

\begin{quest}
    What happens in the side-by-side Sacks model?
\end{quest}

We have noted in the introduction that adding uncountably many Cohen reals results in a model of $\MS(\OD)$. On the other hand, the situation of a single Cohen real simply eludes us. 

\begin{quest}\label{q:Cohen}
    Let $c$ be a Cohen real over $L$. Does $L[c] \models \MS(\OD)$?
\end{quest}

One may show quite easily that $L[c] \models \MS(\Sigma^1_3)$ using the existence of a perfect set of Cohen reals in $L[c]$ and Shoenfield's absoluteness for $\mathbf\Pi^1_2$-formulas. On the other hand, consider that in $L[c]$, the set $A$ of Cohen reals of degree $c$ (that is, the maximal degree) is a $\Pi^1_3$ set disjoint from $L$. One can note that the existence of a perfect subset of $A$ would already entail $\MS(\OD)$. Thus, there are only two options: either $L[c] \models \neg \MS(\Pi^1_3)$ or $L[c] \models \MS(\OD)$.

We haven't touched the following questions which may have similar answers as in Section~\ref{sec:Kanovei}.

\begin{quest}
    Let $n > 2$. Is it consitent with ``$\forall a \in \omega^\omega (V \neq L[a])$" that $\MS(\Sigma^1_n)$ while $\neg \MS(\Pi^1_n)$? What about $\MS(\Pi^1_n)$ versus $\MS(\Sigma^1_n)$? $\MS(\Delta^1_n)$ versus $\neg \MS(\Sigma^1_n) \wedge \neg \MS(\Pi^1_n)$?
\end{quest}

\subsection*{Acknowledgements}
The author was supported in part by a UKRI Future Leaders Fellowship [MR/T021705/2]. This research was funded in part by the Austrian Science Fund (FWF) [10.55776/ESP5711024]. For open access purposes, the author has applied a CC BY public copyright license to any author-accepted manuscript version arising from this submission. No data are associated with this article. 

\bibliographystyle{plain}

\end{document}